\documentclass[a4paper,reqno,12pt]{amsart}

\usepackage[utf8]{inputenc}
\usepackage{amsmath, amssymb, amsthm, epsfig}
\usepackage{hyperref, latexsym}
\usepackage{url}
\usepackage[mathscr]{euscript}
\usepackage{color}
\usepackage{harpoon}
\usepackage{url}
\usepackage{mathabx}
\usepackage{tikz}
\usepackage{cite}

\usepackage{fullpage} 
\usepackage{setspace}
\usepackage{adjustbox}

\usepackage{mathtools}
\mathtoolsset{showonlyrefs}

\def\today{\ifcase\month\or
  January\or February\or March\or April\or May\or June\or
  July\or August\or September\or October\or November\or December\fi
  \space\number\day, \number\year}

\newtheorem{theorem}{Theorem}

\newtheorem{lemma}[theorem]{Lemma}

\newtheorem{definition}{Definition}

\theoremstyle{definition}
\newtheorem{example}{Example}

\newcommand{\I}{\mathcal{I}}

\renewcommand{\P}{\mathrm{P}}

\newcommand{\n}{\mathbb{N}}
\newcommand{\z}{\mathbb{Z}}

\renewcommand{\r}{\mathbb{R}}
\newcommand{\cp}{\mathbb{C}} 

\newcommand{\ft}{\widehat}
\renewcommand{\S}{\mathbb{S}}

\newcommand{\wt}{\widetilde}

\newcommand{\om}{\omega}
\newcommand{\la}{\lambda}
\newcommand{\ga}{\gamma}
\newcommand{\al}{\alpha}

\newcommand{\ep}{\varepsilon}

\newcommand{\si}{\sigma}

\newcommand{\spec}{{\rm spec}}

\begin{document}


\title[]{Sharp Fourier Extension on the Circle under arithmetic constraints}
\author[Ciccone \& Gon\c{c}alves]{Valentina Ciccone  and Felipe Gon\c{c}alves}
\date{\today}
\address{Hausdorff Center for Mathematics, Universit\"at Bonn,
Endenicher Allee 60,
53115 Bonn, Germany }
\email{ciccone@math.uni-bonn.de}
\address{IMPA - Estrada Dona Castorina 110, Rio de Janeiro, RJ - Brasil, 
22460-320
}
\email{goncalves@impa.br}
\allowdisplaybreaks


\begin{abstract}
We establish a sharp adjoint Fourier restriction inequality for the end-point Tomas-Stein restriction theorem on the circle under a certain arithmetic constraint on the support set of the Fourier coefficients of the given function. Such arithmetic constraint is a generalization of a {$B_3$-set}.
\end{abstract}


\maketitle


\section{Introduction}
In this paper we are interested in the optimal constant for the Fourier extension inequality
\begin{align}\label{the_inequality}
||\widehat{f\sigma}||_{L^6(\mathbb{R}^2)}^6\leq \text{C}_{\text{opt}}||f||_{L^2(\mathbb{S}^1)}^6 ~,
\end{align}
where $\sigma$ is the arc length measure on $\mathbb{S}^1$, $\widehat{f\sigma}$ is the Fourier transform of the measure $f\sigma$, 
$$\widehat{f\sigma}(x)=\int_{\mathbb{S}^1}f(\omega)e^{-ix\cdot \omega}d\sigma_\omega,\quad x\in\mathbb{R}^2,$$
and $\text{C}_{\text{opt}}$ is the optimal constant
$$\text{C}_{\text{opt}}:=\sup_{f\in L^2(\mathbb{S}^1),\; f\neq 0} ||\widehat{f\sigma}||_{L^6(\mathbb{R}^2)}^6||f||_{L^2(\mathbb{S}^1)}^{-6} ~.$$
This problem has attracted a lot of attention in the last decade.  
Existence of maximizers have been established in \cite{shao2016existence} and it is known that maximizers are smooth { \cite{shao2016smoothness,diogo_quil_smoothness}}, and that they can be chosen to be  non-negative and antipodally symmetric, see \cite{carneiro2017sharp}. In \cite{carneiro2017sharp}, and later in \cite{GN20}, it has been  established that constant functions are  local maximizers. In fact, it is conjectured that constant functions are indeed global maximizers in which case
\begin{align}\label{optconst}
\text{C}_{\text{opt}}= (2\pi)^4 \int_0^\infty J_0^6(r)rdr ~.
\end{align}
{If this were true, a full characterization of the complex valued maximizers is provided in \cite{carneiro2017sharp}. Moreover, in \cite{OSQ} it is shown that if \eqref{the_inequality} is maximized by constants then the following inequality
\begin{align*}
||\widehat{f\sigma}||_{L^{2k}(\mathbb{R}^2)}\leq \text{C}_{2k,\text{opt}}||f||_{L^2(\mathbb{S}^1)} ~,
\end{align*}
is also maximized by constants for every $k> 3$.} 

A major technical challenge in the study of extremizers for \eqref{the_inequality} lies in the fact that the threefold convolution $\sigma\ast\sigma\ast\sigma(x)$, which arises naturally when exploiting the evenness of the exponent in the right hand side of \eqref{the_inequality} and using Plancharel, blows up when $|x|=1$, {see \cite{carneiro2017sharp}}.

We recall that any complex-valued $f\in L^2(\S^1)$ can be expanded in Fourier series 
$$
f(\om) = \sum_{n\in\z} \ft f(n) \underline{\om}^n,
$$
where we let $\underline{\om}=x+iy$ if $\om=(x,y) \in \S^1$. We also define the spectrum of $f$ to be
$$
\spec(f)=\{n\in \z : \ft f(n)\neq 0\}.
$$
In \cite{oliveira2019band, barker2020band} the case of band-limited functions was explored, that is, when $|\spec(f)|<\infty$. Specifically, it has been shown that {constant functions are the unique maximizers among the class of real-valued, non-negative, antipodally symmetric functions $f\in L^2(\mathbb{S}^1)$ with  $\spec(f) \subseteq [-30,30]$ and $\spec(f) \subseteq [-120,120]$, respectively. Note that when restricting to the band-limited case the problem becomes finite-dimensional (a matter of computing the eigenvalues of a quadratic form) and it can be addressed numerically as done in \cite{oliveira2019band, barker2020band}.}

{In this paper we consider functions in $L^2(\mathbb{S}^1)$  whose spectrum can be infinite, but it satisfies certain arithmetic constraints. More specifically, we establish}
the desired sharp inequality for functions in $L^2(\mathbb{S}^1)$ whose spectrum is sufficiently \textit{sparse} in the following sense:

\begin{definition}\label{property_p3_definition}
A set $A \subset \z$ is said to be a $\mathrm{P}(3)$-set if for every $D\in A + A + A$ one (and only one) of the following {holds}:
\begin{itemize}
    \item $D$ is unique, that is, $D=a_1+a_2+a_3$, the triple $(a_1,a_2,a_3)\in A\times A\times A$ is unique modulo permutations and $a_i\neq - a_j$ for $i\neq j$;
    \item $D$ is trivial, that is, $D$ is not unique and the only way of representing $D$ is $D=D+a-a$ for some $a\in A\cap (-A)$.
\end{itemize}
\end{definition}
\noindent { We use the terms unique and trivial here merely for useful case distinction to be used later on in the proof of our main result. We extend such notion and give a more general definition in Section \ref{Section_NT} for arbitrary $h$-sums ${A+A+...+A}$, what we call $\P(h)$-sets. A complete list of all $\P(3)$-sets $A\subset [-3,3]\cap \z$ 
is }
\begin{align*}
& \{0\}, \ \pm\{1\}, \ \pm\{2\}, \ \,\pm\{3\}, \ \,\{-3,3\}, \ \,\{-2,2\}, \ \,\pm\{-2,3\}, \ \{-1,1\}, \ \pm\{-1,2\}, \\ & \pm\{-1,3\}, \  \pm\{0,1\}, \ \pm\{0,2\}, \pm \{0,3\}, \ \pm \{1,2\}, \ \pm\{1,3\}, \ \pm\{2,3\}, \ \{-3,0,3\}, \\ &  \pm\{-3,2,3\}, \ \pm\{-2,-1,3\}, \ \pm\{-2,0,2\}, \ \pm \{-2,0,3\}, \  \{-2,1,2\},  \ \pm\{-2,1,3\}, \\ & \pm\{-2,2,3\}, \ \{-1,0,1\}, \ \pm\{-1,0,3\}, \ \pm\{-1,2,3\}, \ \{-3,-2,2,3\}.  
\end{align*}
\noindent A simple example of a symmetric infinite $\P(3)$-set is $A=\{ \pm 6^{n}: n\geq 0\} \cup \{0\}$ (see Example \ref{exepower}). 

{ Beside providing explicit examples of constructions of $\P(3)$-sets one may ask how fast can the counting function
$$x\mapsto |\spec(f) \cap [-x,x] |$$
grow.}
In Example \ref{exe7} we construct an infinite symmetric $\P(3)$-set $A$ (via greedy choice) such that 
$$
|A \cap [-x,x] | \gtrsim x^{1/5},
$$
On the other hand, if $A$ is a $\P(3)$-set then it is easy to see that $A\cap [1,\infty]$ and $(-A) \cap [1,\infty]$ are $B_3$-sets (see Section \ref{Section_NT}) and thus if we consider the $\binom {|A_x|+2} 3$ multi-sets of size $3$ in $A_x=A\cap [1,x]$, the sums of the elements represent each number in $[1,3x]\cap \z_+$ at most once, hence
$$
3x \geq \binom {|A_x|+2} 3  = \tfrac16 (|A_x|+2)(|A_x|+1)|A_x|,
$$
and so $|A \cap [-x,x]| \lesssim  x^{1/3}$. Constructing $B_h$-sets with large {density} is a very hard task that have drawn a lot of attention in the literature, especially in the interplay of combinatorics, probability and number theory, and we refer to the introduction of \cite{cilleruelo2014infinite} and the references therein for further information. Since any $B_3$-set is a $\P(3)$-set  we can simply {
rely on Cilleruelo's 
result \cite{cilleruelo2014infinite}, see also Cilleruelo and Tesoro \cite{cilleruelo2015dense}, to obtain existence of} a $\P(3)$-set $A$ with only positive integers and counting function satisfying
$$
|A \cap [-x,x] | \gtrsim x^{\sqrt{5}-2},
$$
which is the current best 
{existence result}
in terms of the exponent $\sqrt{5}-2=0.23...$.

We are now ready to state our main result.

\begin{theorem}\label{theroem_main_result}
Let $f\in L^2(\mathbb{S}^1)$ be such that its spectrum
$$
\spec(f)=\{n\in \z : \ft f(n)\neq 0\}
$$
is a $\P(3)$-set. Then
$$||\widehat{f\sigma}||_{L^6(\mathbb{R}^2)}^6\leq (2\pi)^4 \bigg(\int_0^\infty J_0^6(r)rdr\bigg) ||f||_{L^2(\mathbb{S}^1)}^6.$$
and equality is attained {if and only if} $f$ is constant.
\end{theorem}

To the best of our knowledge it is the first time that this inequality is established for functions $f$ with an infinite spectrum and that simultaneously do not need to be {``close''} to constant functions. 
{A simple function which is not comtemplated by the previous results \cite{carneiro2017sharp,oliveira2019band, barker2020band} is $f(\underline{\om})= 1 + c\underline{\om}^n$ for arbitrary $n$ and large $c$. The function $f$ is not real-valued nor non-negative antipodally symmetric, hence we cannot apply the results in \cite{oliveira2019band, barker2020band}. Moreover, for large $c$ we cannot apply the local result of \cite{carneiro2017sharp}. However, the set $A=\{0,n\}$ is a $\P(3)$-set, and thus Theorem \ref{theroem_main_result} applies. }

Clearly, by translation invariance, one could instead ask that the spectrum of $\om \mapsto e^{i \tau \cdot \om} f(\om)$ is a $\P(3)$-set for some $\tau\in \r^2$ and obtain the same inequality. 

{The proof of Theorem \ref{theroem_main_result} relies crucially on some refined estimates on integrals involving the product of six Bessel functions. Some of these integrals involve Bessel functions of lower order and need to be estimated numerically.}
In Lemma \ref{lemma_bounds_on_bessel} we estimate such integrals by employing a new method (quite different from \cite{e2017estimates,oliveira2019band, barker2020band}) that avoids doing any numerical integration, and makes use instead of a known quadrature formula for band-limited functions in $\r^2$.

\subsection{Overview}
This paper is organized as follows. In Section \ref{Section_NT} we give a precise definition of {$\mathrm{P}(h)$-set}. Then we propose some examples of 
{$\mathrm{P}(h)$-set} with non-trivial symmetric subsets. 
In Section \ref{Section_BF} we study some refined estimates on integral involving the product of six Bessel functions. In Section \ref{Section_MR} we prove our main result. Finally, in Section \ref{extra_section} we propose a further example of application of the developed strategy to the study of sharp inequalities.

\section{A Generalization of {$B_h$-sets} }\label{Section_NT}
{ A subset $S\subseteq\mathbb{Z}$ is said to be a $B_h$-set, with $h\geq 2$, if for any $a_1,...,a_h,b_1,...,b_h\in S$ such that $a_1+...+a_h=b_1+...+b_h$ we have that $(a_1,...,a_h)$ is a permutation of $(b_1,...,b_h)$. If $h=2$ the set $S$ is sometimes said to be a Sidon set \footnote{  This has not to be confused with the other definition of Sidon set according to which a set $E$ is a Sidon set if every continuous function $f:\S^1\to\cp$ with $\spec(f)\subseteq E$ has absolutely convergent Fourier series.  To avoid confusion we will always refer to $B_h$-sets with $h=2$ as $B_2$-sets.}.
We are interested in defining a suitable generalization of $B_h$-sets to account for the case of sets $A\subseteq\mathbb{Z}$ with non-trivial symmetric subsets, namely such that $|A\cap -A|\geq 3$.
 It is immediate to see that such symmetric sets cannot be $B_h$-sets: in fact, for example, when $h$ is even there is always more than one way of representing zero as sum of $h$ elements in $A$, whereas when $h$ is odd there is always more than one way of representing any element in $A$ as a sum of elements in $A$.}

In what follows we let $A^k$ denotes the iterated sum of $k$ copies of $A$, e.g. $A^3=A+A+A$. 

\begin{definition}[Property $\mathrm{P}(h)$]  \label{property_p_definition}
We say that the set $A$ satisfies property $\mathrm{P}(h)$ {(with $h\geq 2$), or that $A$ is a $\P(h)$-set,}  if for any $D\in A^h$  there exists 
$0\leq \ell \leq h$ with the same parity of $h$ and a unique set of $\ell$ elements $\lbrace a_1, ..., a_\ell\rbrace$, with $a_1,...,a_\ell\in A$, $a_i\neq - a_j$ for all $i\neq j$, and such that any $h-$tuple $(b_1,...,b_h)$, with $b_1,...,b_h\in A$ and $b_1+...+b_h=D$ is a permutation of a $h-$tuple $(a_1,...,a_\ell, {u_1, -u_1, ... , u_{(h-l)/2}, -u_{(h-l)/2}})$ for some $u_1,...,u_{(h-l)/2}\in A \cap (-A)$.
\end{definition}

We recall that a set $E\subset\mathbb{Z}$ is said to be a $\Lambda_p$ set, for some $p>2$, if there exists a constant $C$ such that $$||f||_{L^p(\mathbb{S}^1)}\leq C ||f||_{L^2(\mathbb{S}^1)} $$
for all functions $f\in L^2(\mathbb{S}^1)$ whose spectrum is contained in $E$. It is well known  that $B_h$-sets are $\Lambda_{2h}$ sets (see e.g. \cite{bourgain2001ap}).

The following observations follow immediately from the definition of property $\mathrm{P}(h)$.

\begin{itemize}
\item  If $A$ is a $\P(h)$-set then $A \cap \n$ and $-A \cap \n$ are $B_h$-sets, and thus $A$ is a $\Lambda_{2h}$ set.
\item If $A$ is a $B_h$-set then $A$ is a $\P(h)$-set.
    \item If $|A\cap -A| \leq 2$ and $A$ is a $\mathrm{P}(h)$-set then $A$ is a $B_h$-set. 
    \item If $A$ is a $\P(h)$-set and $S\subseteq A$ then $S$ is a $\P(h)$-set.
     \item If $A$ is a $\P(h)$-set then $-A$ is a $\P(h)$-set.  
    \item If $A$ is a $B_h$-set, the set $A\cup -A$ does not necessarily satisfy property $\mathrm{P}(h)$: in fact, for example, the set of powers of two, $\lbrace 1,2,4,... \rbrace$, is a $B_2$-set, however the set  $\lbrace -1,-2,-4,... \rbrace \cup  \lbrace 1,2,4,... \rbrace $ does not satisfy property $\mathrm{P}(2)$, since, for example, $1+1=4-2$.
\end{itemize}

\subsection{Examples of $\P(h)$-sets}

Since any $B_h$-set is a $\P(h)$-set, the more interesting task is to provide examples of sets $A$ that satisfy property $\mathrm{P}(h)$ for some $h$ and that are such that $|A\cap -A| \geq 3$.

\begin{example} 
A sequence of positive integers $\lbrace \lambda_n \rbrace$ is said to be (Hadamard) lacunary if  $\lambda_{n+1}\geq q\lambda_n$ for some $q>1$.  Let $A_{\lambda,q}:=\lbrace \lambda_n\rbrace \cup (-\lbrace \lambda_n \rbrace) \cup \lbrace 0 \rbrace$. We claim that $A_{\lambda,q}$ satisfies property $\mathrm{P}(h)$ whenever $q> 2h-1$. To see this assume $a_1,...,a_h,b_1,...,b_h \in A_{\lambda,q}$ are such that \begin{align}\label{proof_first_equ}
a_1+...+a_h=b_1+...+b_h ~.
\end{align} 
On both sides of \eqref{proof_first_equ} we omit the zero terms and simplify terms of the form $a_j+a_i$ with $a_i=-a_j$ and $b_m+b_n$ with $b_m=-b_n$. {If no term is left on both sides of \eqref{proof_first_equ} then $a_1+...+a_h=b_1+...+b_h=0$ and $(a_1,...,a_h)$, $(b_1,...,b_h)$ are consistent with property $\P(h)$. On the other hand, if terms are left on at least one side of \eqref{proof_first_equ} we} further arrange them so that to have only positive terms on both sides obtaining
\begin{equation}\label{eq_after_rearrangement}
\alpha_{1}+...+\alpha_{h'}= \beta_{1}+...+\beta_{h''}
\end{equation}
where $\lbrace\alpha_1,...,\alpha_{h'},\beta_1,...,\beta_{h''}\rbrace\subseteq { \lbrace |a_1|,...,|a_h|,|b_1|,...,|b_h| \rbrace }$, $\alpha_1,...,\alpha_{h'},\beta_1,...,\beta_{h''} > 0$ and $h'+h''\leq 2h$, $h',h''\geq 1$. We want to show that $\lbrace \alpha_1,...,\alpha_{h'} \rbrace= \lbrace \beta_1,..., \beta_{h''}\rbrace$. We proceed in a similar way as in \cite[Proof of Theorem  3.6.4.]{loukas2014classical}. We start by showing that $\max\lbrace \alpha_1,...,\alpha_{h'} \rbrace= \max\lbrace \beta_1,..., \beta_{h''}\rbrace$. Assume by contradiction that $\max\lbrace \alpha_1,...,\alpha_{h'} \rbrace > \max\lbrace \beta_1,..., \beta_{h''}\rbrace$. Then $\max\lbrace \alpha_1,...,\alpha_{h'} \rbrace \geq q \max\lbrace \beta_1,..., \beta_{h''}\rbrace$. On the other hand we have
$$ \max\lbrace \alpha_1,...,\alpha_{h'} \rbrace \leq \beta_1+...+\beta_{h''}\leq h'' \max\lbrace \beta_1,...,\beta_{h''} \rbrace < q \max\lbrace \beta_1,...,\beta_{h''} \rbrace  $$
where in the last inequality we have used the fact that $h''\leq 2h-1 < q$. By assuming that $\max\lbrace \beta_1,..., \beta_{h''}\rbrace > \max\lbrace \alpha_1,...,\alpha_{h'} \rbrace$ we have a similar contradiction. Hence 
 $\max\lbrace \alpha_1,...,\alpha_{h'} \rbrace= \max\lbrace \beta_1,..., \beta_{h''}\rbrace$. Proceeding by induction we see that $h'=h''$ and $\lbrace \alpha_1,...,\alpha_{h'} \rbrace= \lbrace \beta_1,..., \beta_{h''}\rbrace$ as claimed. 
 {Because we got rid of the cases $a_i=-a_j$, $b_m=-b_n$ in the very beginning this further implies that there exists a unique (up to permutation) $h'$-tuple of elements in $A_{\lambda,q} $ that sums up to $D=a_1+...+a_h=b_1+...+b_h$ .} 
\end{example}

\begin{example} \label{exepower}As a particular case of the above result, the set $S_q:=\lbrace \pm q^n:  n\geq 0 \rbrace \cup \lbrace 0 \rbrace$ is a $\mathrm{P}(h)$-set whenever $q\geq 2h$. Note that the set $S_{q}$ for $q=2h-1$ does not satisfy property $\mathrm{P}(h)$ since 
$$
hq=\underbracket{q+...+q}_{h \text{ times}} = q^2 + \underbracket{(-q-...-q)}_{h-1 \text{ times}}
$$
It begs the question whether we can still prove Theorem \ref{theroem_main_result} with an adaptation of our method for the functions $f$ with $\spec(f)\subset S_q$ for $q=5,4,3,2$. We leave this question for future work. One interesting and possibly useful feature is that for $q=5,4,3$ we only have finitely many exceptions (modulo multiplication by $q^n$) breaking property $\P(3)$. For instance, for $q=3,4,5$ the only exceptions are { $1+1+1=9-3-3=3+0+0, \, 1+1+1=4-1-0, \, 1+1+1=5-1-1$}. In generality, for $h\leq q <2h$ the set $S_q$ only has finitely many exceptions not satisfying the property $\P(2h)$. To see this, let $2\leq h\leq q$, $b\in \z$ and $m\geq 2$ with $|b|+m\leq 2h+1$. We claim there are only finitely many solutions to
\begin{align}\label{beq}
b=\sum_{j=1}^{m-1} a_j q^{l_j}
\end{align}
with $a_j=\pm 1$, $l_i \geq l_{i+1}\geq 0$ and with the property that $l_i=l_j$ implies $a_i=a_j$. Such claim with $b=0$ easily shows what we want. Note that if $l_{m-1}>0$ then $q$ divides $b$ and so $b=q$, $l_{m-1}=1$, and we obtain $1\pm 1 = \sum_{j=1}^{m-2} a_j q^{l_j-1}$. If $l_{m-1}=0$ then $1\pm 1 = \sum_{j=1}^{m-2} a_j q^{l_j}$. In any case, if we let $P_{m,b}=\{(a_j,l_j)_{j=1}^{m-1}: \text{solves } \eqref{beq}\}$ we deduce that 
$$
|P_{m,b}| = |P_{m-1,b-1}|+|P_{m-1,b+1}|.
$$
Now note that by unique expansion in base $q$ the set $P_{h,b}$ is a singleton for $|b|\leq  h$ except when $b=0$, in which case $P_{h,b}=\emptyset$, or $|b|=q=h$, in which case  $|P_{h,b}|=2$. This shows that $|P_{m,b}|<\infty$. The case $q=h-1$ has infinitely many exceptional cases, such as: $1+(h-1)^n+...+(h-1)^n = (h-1)^{n+1}+1+0$.

\end{example}

\begin{example} \label{exe7} For a given set $E$ let $E(x)$ be the counting function $E(x):=|E\cap [-x,x]|$.
It is easy to check that the above examples are such that $A_{\lambda, q}(x) \lesssim \log_{2h-1}(x)$ and $S_q(x)\sim \log_q(x) $. An example of a denser set that satisfies property $\mathrm{P}(h)$ can be straightforwardly constructed applying the following greedy algorithm {that generalizes the one of Erd{\"o}s for $B_2$ sets, see \cite{erdHos1981solved,cilleruelo2014infinite}. We start by setting $a_1:=1$ and $a_{-1}:=-a_1$}. Then we define the element $a_n$ to be the smallest integer greater than $a_{n-1}$ and such that the set $\lbrace -a_n,a_{-n+1}, ..., a_{-1}, a_1,...a_{n-1}, a_n \rbrace$ satisfies property $\mathrm{P}(h)$. Then we set {$a_{-n}:=- a_n$} and iterate the procedure. It is easy to check that the resulting sequence of integers $A$ is such that $A(x) \gtrsim x^{1/(2h-1)}$. In fact at each step there cannot be more than $(2n-2)^{2h-1}$ distinct elements of the type $a_{i_1}+...+ a_{i_h} - a_{j_1}-...-a_{j_{h-1}}$ with $-(n-1)\leq i_1, ..., i_h, j_1,..., j_{h-1} \leq n-1$ and therefore $a_n \leq (n-1)^{2h-1} +1$ and  $A(x) \gtrsim x^{1/(2h-1)}$. 
\end{example}

\begin{example} The construction of the following example is adapted from \cite{rudin1960trigonometric} (see also \cite{hewitt1959some}) where a similar strategy is used to construct Sidon sets that are not (Hadamard) lacunary.   For $n=0,1,2,...$ we set $N:=2^n$. Then we define $A_h$ to be the set of elements of the type
$$\pm ((2h)^{4N}+(2h)^{N+j}) \qquad j=0,...,N-1, \; n=0,1,2,...\, .$$
Such a set is not of the type of Example 1, in fact $A_h$ contains $N$ elements between $a_{n,0}:=((2h)^{4N}+(2h)^{N+0})$ and $a_{n,N-1}:= ((2h)^{4N}+(2h)^{N+N-1}) $, $a_{n,N-1} < 2 a_{n,0} $ ,  while sets of the type in Example 1 contain a bounded number of elements between $x$ and $2x$ as $x $ tends to infinity. We claim that the set $A_h$ satisfies property $\mathrm{P}(h)$. To see that, 
let $a_{n_i,j_i}, b_{n_i,j_i}\in A_h, \; i=1,...,h$, be such that 
$$a_{n_1,j_1}+ ... + a_{n_h,j_h} = b_{n_1,j_1}+...+ b_{n_h,j_h}.$$
After simplifying on both sides the elements of the type $x + (-x) $ and rearranging we obtain something like
\begin{equation} \label{sum_equal_zero}
    \alpha_1 c_{n_1,j_1} + ... + \alpha_s c_{n_s,j_s} =0
\end{equation}
where $|\alpha_i|\leq 2h-1$, $s\leq 2h$, $c_{n_i, j_i}\in A_h, \; i=1,...,s$,  and we assume that $ c_{n_1,j_1}< ... < c_{n_s,j_s}$. But then if $\al_1\neq 0$ we would have that $c_{n_1,j_1} $ is divisible by a lower power of $2h$ than $c_{n_2,j_2}, ..., c_{n_s,j_s}$ and therefore \eqref{sum_equal_zero} is impossible.
\end{example}

\section{Estimates for certain integrals of Bessel functions}\label{Section_BF}

A simple computation (using the integral representation for Bessel functions) shows that if we let $\underline{\om}=x+iy$ for $\om=(x,y)\in \r^2$ then
$$\widehat{\underline{\om}^n\sigma}(x,y)=2\pi (-i)^n \frac{J_n(\sqrt{x^2+y^2})}{(x^2+y^2)^{n/2}}  \underline{\om}^n, $$
where  $\si$ is the arc measure on $\S^1$ and $J_n$ is the Bessel function of first kind. The following lemma is crucial for the proof of our main result. In what follows, all numerical computations were performed using the \cite[version 2.15.3]{gp} computer algebra system.

\begin{lemma}\label{lemma_bounds_on_bessel}
We have the following inequalities:
\begin{enumerate}
    \item[(i)] For all {integers} $n>1$ it holds that
    \begin{align*}
        \int_0^\infty J_0^4(r)J_n^2(r)rdr< \frac{1}{5}\int_0^\infty J_0^6(r)rdr
    \end{align*}
    and
    $$
     \int_0^\infty J_0^4(r)J_1^2(r)rdr= \frac{1}{5}\int_0^\infty J_0^6(r)rdr.
    $$
    \item[(ii)]  For all {integers} $n>0$ it holds that
    \begin{align*}
        \int_0^\infty J_0^2(r)J_n^4(r)rdr< \frac{2}{15}\int_0^\infty J_0^6(r)rdr.
    \end{align*}
   \item[(iii)] For all {integers} $n>0$ it holds that
    \begin{align*}
        \int_0^\infty J_n^6(r)rdr < \frac{1}{3}\int_0^\infty J_0^6(r)rdr.
    \end{align*}
    \item[(iv)] For all {integers} $n,m >0$, $n\neq m$ and such that $(n,m)\neq (1,2)$ it holds that
    \begin{align*}
        \int_0^\infty J_n^4(r)J_m^2(r)rdr< \frac{1}{9}\int_0^\infty J_0^6(r)rdr.
    \end{align*}
    Moreover, 
    \begin{align*}
        \int_0^\infty J_1^4(r)J_2^2(r)rdr< \frac{1}{6}\int_0^\infty J_0^6(r)rdr.
    \end{align*}
     \item[(v)] For all {integers} $n>m>\ell\geq 0$ such that $(n,m,\ell)\neq (3,2,0)$ it holds that
    \begin{align*}
        \int_0^\infty J_\ell^2(r)J_m^2(r)J_n^2(r)rdr< \frac{1}{15}\int_0^\infty J_0^6(r)rdr.
    \end{align*}
    Moreover, 
    \begin{align*}
        \int_0^\infty J_3^2(r)J_2^2(r)J_0^2(r)rdr< \frac{1}{6}\int_0^\infty J_0^6(r)rdr.
    \end{align*}
    \end{enumerate}
\end{lemma}

We start by recalling some known bounds on Bessel functions and Bessel integrals.

\begin{itemize}
    \item The following pointwise bound can be found in \cite[Corollary~9]{e2017estimates}
    \begin{align}\label{O&S_T_bound_J0}
       r^{1/2} |J_0(r)|< \gamma
    \end{align}
    for $r>0$, where $\gamma=0.89763 $ (which is a truncation of $\tfrac{9}{8}\sqrt{ \frac{2}{\pi} })$.
    \item The following pointwise bound was proven in \cite{olenko2006upper}
    \begin{align}\label{olenko}
       {r}^{1/2}|J_n(r)|< \beta \sqrt{n^{1/3}+\tfrac{\alpha}{n^{1/3}}+\tfrac{3\alpha^2}{10n}}
    \end{align}
     for $r>0$ and $n>0$  where $\beta=0.674886$ and $\alpha=1.855758$.
      \item  The following identity can be found in \cite[Equation~6.574-2]{gradshteyn2014table}
    \begin{align}\label{gradshteyntable}
        \int_0^\infty J_n^2(r)r^{-\lambda}dr= \frac{\Gamma(\lambda)\Gamma(n+\tfrac{(1-\lambda)}{2})}{2^\lambda\Gamma\big( \tfrac{1+\lambda}{2} \big)^2 \Gamma\big( n + \tfrac{1+\lambda}{2}\big)}
    \end{align}
    for $0<\la<2n+1$.
\end{itemize}
We are now ready to prove the lemma. For simplicity we define
\begin{align}\label{Inotation}
I(n_1,...,n_6):=\int_0^\infty J_{n_1}(r)...J_{n_6}(r)rdr  \ \text{ and } \ \I(n_1,n_2,n_3):=I(n_1,n_1,n_2,n_2,n_3,n_3).
\end{align}
\proof
First we require a good lower bound for $\I(0,0,0)$ which is
$$
\I(0,0,0)> 0.33682.
$$
Such numerical lower bounds (and the upper bounds below) were done using Lemma \ref{lemnumerical} and evaluating the sum $\wt \I(n,m,\ell)$ with high precision (nowadays most computer algebra systems can do it extremely fast).

We start by proving the estimate in (i). Using \eqref{O&S_T_bound_J0} and \eqref{gradshteyntable}
we obtain
\begin{align}\label{first_case}
\begin{split}
    \I(0,0,n)  < \ga^4 \int_0^\infty J_n^2(r)r^{-1} dr
    =  \ga^4 \frac{1}{2n}.
\end{split}
\end{align}
One can easily check that $\ga^4 \tfrac{1}{2n}\leq \tfrac{1}{5}0.33682$ when $n\geq 5$. For $n=2,3,4$ we have
$$
\I(0,0,n) < \wt \I(0,0,n) + 10^{-4}\leq \wt \I(0,0,2) +  10^{-4} = 0.0370...,
$$
which is visibly less than $\tfrac{1}{5}0.33682=0.067364$. Integration by parts in conjunction with the relation $J_0(r)=J_1'(r)+\tfrac{1}r J_1(r)$ shows the desired identity $\I(0,0,0)=5\I(0,0,1)$.

To prove item (ii) we use \eqref{O&S_T_bound_J0}, \eqref{olenko} and \eqref{gradshteyntable} to obtain
\begin{align*}
\begin{split}
    \I(0,n,n) <  \ga^2 \beta^2 \bigg( n^{1/3}+\frac{\alpha}{n^{1/3}} +\frac{3\alpha^2}{10 n} \bigg) \int_0^\infty J_n^2(r)r^{-1}dr =  \ga^2 \beta^2 \bigg( n^{1/3}+\frac{\alpha}{n^{1/3}} +\frac{3\alpha^2}{10 n} \bigg)\frac{1}{2n}.
\end{split}
\end{align*}
The right hand side above is a decreasing function of $n$ and one can check that it is less than $\tfrac{2}{15}0.33682$ for $n\geq 14 $. For $1\leq n \leq 13$ we have the numerical bounds
$$
\I(0,n,n) < \wt \I(0,n,n) +  10^{-4} \leq  \wt \I(0,1,1) + 10^{-4} = 0.0424...,
$$
which is less than $\tfrac{2}{15}0.33682=0.0449093...$.

To prove item (iii) we use \eqref{olenko} and \eqref{gradshteyntable} to obtain
\begin{align*}
    \begin{split}
        \I(n,n,n)  <  \bigg( \beta \sqrt{n^{1/3}+\frac{\alpha}{n^{1/3}}+\frac{3\alpha^2}{10n}} \bigg)^4 \frac{1}{2n}.
    \end{split}
\end{align*}
Also in this case the right hand side is a decreasing function of $n$. When $n\geq 9$ the right hand side is less than $\tfrac{1}{3}0.33682$, while for $1\leq n \leq  8$ we have the numerical bounds
$$
\I(n,n,n) < \wt \I(n,n,n) + 10^{-4}\leq  \wt \I(1,1,1) + 10^{-4}= 0.1049...,
$$
which is less than $\tfrac{1}{3}0.33682=0.112273...$.

To prove the estimate in (iv) first let for $n\geq 1$
$$
B_n = \beta \sqrt{n^{1/3}+\frac{\alpha}{n^{1/3}}+\frac{3\alpha^2}{10n}}
$$
and $B_0=\ga$. One can show that $B_n$ is increasing for $n\geq 6$ and $\max_{0\leq n\leq 5} B_n = B_{1}<B_{36}$.  Next we let $k=\max\{n,m\}$ and we use 
\eqref{olenko} and \eqref{gradshteyntable} to obtain
\begin{align*}
\begin{split}
   \I(n,n,m) < \max\{B_{36}^2,B_k^2\}^2  \int_0^\infty J_k^2(r) r^{-1}dr = \frac{\max\{B_{36}^2,B_k^2\}^2}{2k}.
\end{split}
\end{align*}
The right hand side is a decreasing function of $k$ and one can easily check that it is less than $\tfrac{1}{9}0.33682$ when $k\geq 49 $. For $1\leq k\leq 48$ with $(n,m)\neq (1,2)$ we rely on the numerical bounds
$$
\I(n,n,m) <  \wt \I(n,n,m) + 10^{-4} \leq \wt \I(2,2,3) + 10^{-4} = 0.0335...,
$$
which is visibly less than $\tfrac{1}{9}0.33682=0.037424...$.
Moreover,
$\I(1,1,2)<\wt \I(1,1,2) + 10^{-4}=0.0424... < \tfrac{1}{6}0.33682=0.0561...$.

To prove the estimate in (v) we can then use \eqref{olenko} and \eqref{gradshteyntable} to obtain
\begin{align*}
  \I(n,m,\ell) & < B_\ell^2 B_m^2 \int_0^\infty J_n^2(t) r^{-1}dr = \frac{B_\ell^2 B_m^2}{2n} \leq \frac{\max\{B_{36}^2,B_{\ell}^2\}\max\{B_{36}^2,B_{m}^2\}}{2n} \\ & \leq \frac{\max\{B_{36}^2,B_{n-2}^2\}\max\{B_{36}^2,B_{n-1}^2\}}{2n}.
\end{align*}
A tedious computation shows again that the right hand side above is indeed a decreasing function of $n$. One can easily check that it is less than $\tfrac{1}{15}0.33682$ when $n\geq 145 $, while for $1\leq n \leq 144$ with $(n,m,\ell)\neq (3,2,0)$ we have the numerical bounds
$$
\I(n,m,\ell) <  \wt \I(n,m,\ell) + 10^{-4} \leq \wt \I(4,2,0) + 10^{-4} = 0.0185...,
$$
which is less than $\tfrac{1}{15}0.33682=0.0224546...$. Moreover,
$\I(3,2,0)<\wt \I(3,2,0) +10^{-4}=0.0243... < \tfrac{1}{6}0.33682=0.0561...$.
\qed
  
 \begin{lemma}\label{lemnumerical}
 For all { integers} ${ k},m,\ell \geq 0$ with $\max\{{ k},m,\ell\}\leq 11519$ we have
 $$
\wt \I({ k},m,\ell) < \I({ k},m,\ell) < \wt \I({ k},m,\ell)  +  10^{-4},
 $$
 where
 $$
\wt  \I({ k},m,\ell)=\frac{2}{9}\sum_{n=0}^{22000} \frac{J_{{ k}}(\la_n/3)^2J_{m}(\la_n/3)^2J_{\ell}(\la_n/3)^2}{J_0(\la_n)^2}
 $$
 and $\{\la_n\}_{n\geq 0}$ are the nonnegative zeros of the Bessel function $J_1$ (with $\la_0=0$) .
  \end{lemma}
  
 \begin{proof}
First we use a particular case of a formula of Ben Ghanem and Frappier \cite{GhFr98} (although this identity can be found  in disguise in much older papers) which says that if $f\in L^1(\r^2)$ is radial and $\text{supp}( \ft f) \subset B_{\tfrac1{\pi}}(0)$ (or equivalently, if $f$ is analytic in $\cp^2$ and has exponential type at most $2$) then
$$
\int_{\r^2} f(x)dx = {4\pi}\sum_{n\geq 0} \frac{f(\la_n)}{J_0(\la_n)^2}.
$$
We can apply this formula for $f(x)=J_{{ k}}(\tfrac13|x|)^2J_{m}(\tfrac13|x|)^2J_{\ell}(\tfrac13|x|)^2$ to deduce that
$$
\I({ k},m,\ell)=\frac{2}{9}\sum_{n\geq 0} \frac{J_{{ k}}(\la_n/3)^2J_{m}(\la_n/3)^2J_{\ell}(\la_n/3)^2}{J_0(\la_n)^2}.
$$
We then let $\wt \I({ k},m,\ell)$ be the above sum truncated up to $n=22000$. To bound the tail we first apply  \cite[Theorem 3]{Kra14}, from which one easily deduce that
    \begin{align}\label{krasikov}
       r^{1/2} |J_{ k}(r)|< 1
    \end{align}
    for all ${ k}\geq 1$ and $r>2{ k}$. Noticing that $\la_{22001}/3=23039.65... > 2 \times 11519$ we obtain
$$
\frac{2}{9}\sum_{n > 22000} \frac{J_{{ k}}(\la_n/3)^2J_{m}(\la_n/3)^2J_{\ell}(\la_n/3)^2}{J_0(\la_n)^2} < {6}\sum_{n > 22000} \frac{1}{\la_n^3 J_0(\la_n)^2}.
$$
Secondly, we apply Krasikov's effective envolope \cite[Lemma 1]{Kra06} for $\nu=0$ (noting that $\mu=3$ and $J_0'(x)=-J_1(x)$) to obtain that
$$
J_0(\la_n)^2 > {0.99 \times \frac{2}{\pi \la_n}}
$$
for $n> 22000$ (indeed $J_0(\la_n)^2 \sim \frac{2}{\pi \la_n}$).
Now we apply a result of Makai \cite{Ma78} that shows that $\nu\mapsto \la_{\nu,n}/\nu$ is decreasing, where $\la_{\nu,n}$ is the $n$-th zero  of $J_\nu$.  It is also well-known that $\la_{\nu+1,n}>\la_{\nu,n}$ for all $n\geq 1$ and $\nu>-1$. Hence
$$
\la_n =\frac{\la_{1,n}}{1}  > \frac{\la_{3/2,n}}{3/2} > \frac{\la_{1/2,n}}{3/2} = \tfrac{2}3\pi n,
$$
because $J_{1/2}(x) = \sqrt{2\pi/x}\sin(x)$ (with a more careful search in the literature one could possibly derive $
\la_n\geq .99 \pi n$ for $n> 22000$, since $\la_n \sim \pi n$). We obtain that
$$
{6}\sum_{n > 22000} \frac{1}{\la_n^3 J_0(\la_n)^2} < 2.18 \sum_{n > 22000} \frac{1}{n^2} < 2.18 \int_{22000}^\infty x^{-2}dx < 10^{-4}.
$$
\end{proof}

\section{Proof of the main result}\label{Section_MR}

Let $f\in L^2(\mathbb{S}^1)$ be a complex valued function and let $A=\spec(f)$. Then by Hecke-Bochner formula we can write
\begin{align*}
 (2\pi)^{-7}||\widehat{f\sigma}||_{L^6(\mathbb{R}^2)}^6 & = (2\pi)^{-7} \int_{\mathbb{R}^2} \widehat{f\sigma}(x)\widehat{f\sigma}(x)\widehat{f\sigma}(x)\overline{\widehat{f\sigma}(x)\widehat{f\sigma}(x)\widehat{f\sigma}(x)}dx\\
    &= \sum_{\substack{n_1,...,n_6\in A\\ n_1+n_2+n_3=n_4+n_5+n_6}} \widehat{f}(n_1)\widehat{f}(n_2)\widehat{f}(n_3)\overline{\widehat{f}(n_4)\widehat{f}(n_5)\widehat{f}(n_6)} I(n_1,...,n_6)\\
    &= \sum_{D\in A^3}\sum_{\substack{n_1,...,n_6\in A\\ n_1+n_2+n_3=D\\ n_4+n_5+n_6=D}} \widehat{f}(n_1)\widehat{f}(n_2)\widehat{f}(n_3)\overline{\widehat{f}(n_4)\widehat{f}(n_5)\widehat{f}(n_6)} I(n_1,...,n_6)
\end{align*}
where we are using {the notation introduced in} \eqref{Inotation}. Now if $A$ satisfies $\mathrm{P}(3)$ we can split the last summation over the $D\in A^3$ that are unique and over those that are trivial. Note that by Definition \ref{property_p3_definition} if $0\in A^3$ then $0$ is trivial. We focus first on the case $D\in A^3$ that are unique for which we obtain the following
\begin{align*}
    (I): = & \sum_{\substack{D\in A^3\\ D \, unique}}\sum_{\substack{n_1,n_2,n_3,n_4,n_5,n_6\in A\\ n_1+n_2+n_3=D\\ n_4+n_5+n_6=D}} \widehat{f}(n_1)\widehat{f}(n_2)\widehat{f}(n_3)\overline{\widehat{f}(n_4)\widehat{f}(n_5)\widehat{f}(n_6)} I(n_1,n_2,n_3,n_4,n_5,n_6) \\
= & 6 \sum_{\substack{n_1,n_2,n_3\in A \\ |n_i|\neq |n_j| \text{ for } i\neq j}}  |\widehat{f}(n_1)|^2|\widehat{f}(n_2)|^2|\widehat{f}(n_3)|^2  \I(n_1,n_2,n_3)\\
& + 9 \sum_{\substack{n_1 \in A\setminus\lbrace 0\rbrace, \, n_3\in A\\ |n_1|\neq |n_3| }}  |\widehat{f}(n_1)|^4|\widehat{f}(n_3)|^2  \I(n_1,n_1,n_3)\\
& +   \sum_{n_1\in A\setminus\lbrace 0\rbrace}  |\widehat{f}(n_1)|^6 \I(n_1,n_1,n_1),
\end{align*}
Now we focus on the sum over the set  $\lbrace D \in A^3: \, D \, \text{ trivial} \rbrace$. We use the short hand notation $A_{s}=A\cap (-A)$, $A_t=A \cap \text{trivial}$ and $A_{s,t}=A_s \cap A_t$. In this case the set
$\lbrace (n_1,n_2,n_3)\in A\times A\times A: \, n_1+n_2+n_3=D \rbrace $
is the disjoint union of the following sets
\begin{align*}
  S_1(D)=  & \lbrace (D, a, -a): a\in A_s\setminus \lbrace \pm D \rbrace \rbrace\\
S_2(D)=    & \lbrace (-a, D, a): a\in A_s\setminus \lbrace \pm D \rbrace \rbrace\\
S_3(D)=    & \lbrace (-a, a, D): a\in A_s\setminus \lbrace \pm D \rbrace \rbrace\\
S_4(D)=    & \lbrace (D, D, -D), (-D, D, D), (D,- D, D) \rbrace.
\end{align*}
Letting $\ep_D=|S_4(D)|$ we obtain
\begin{align*}
    (II):= & \sum_{\substack{D\in A^3\\ D \; trivial}} \sum_{i,j=1}^4 \sum_{\substack{(n_1,n_2.n_3)\in S_i(D)\\ (n_4,n_5.n_6)\in S_j(D)}} \widehat{f}(n_1)\widehat{f}(n_2)\widehat{f}(n_3)\overline{\widehat{f}(n_4)\widehat{f}(n_5)\widehat{f}(n_6)} I(n_1,...,n_6) \\
   = & 9 \sum_{\substack{D\in A_t  \\ n_1,n_2\in A_s\setminus \lbrace \pm D \rbrace}} \widehat{f}(D)\widehat{f}(n_1)\widehat{f}(-n_1)\overline{\widehat{f}(D)\widehat{f}(n_2)\widehat{f}(-n_2)} I(D,n_1,-n_1,D,n_2,-n_2)\\
   & + 18 \, \mathfrak{Re} \bigg( \sum_{\substack{D\in A_{s,t}\setminus \lbrace 0 \rbrace \\ n_1\in A_s\setminus \lbrace \pm D \rbrace} }  \; \widehat{f}(D)\widehat{f}(D)\widehat{f}(-D)
   \overline{\widehat{f}(n_1)\widehat{f}(-n_1)\widehat{f}(D)} I(D,D,-D,D, n_1,-n_1) \bigg) \\
   & + 6 \, \mathfrak{Re}\bigg(\widehat{f}(0)\widehat{f}(0)\widehat{f}(0)\;  \sum_{n_1\in A_s\setminus \lbrace 0 \rbrace} \overline{\widehat{f}(n_1)\widehat{f}(-n_1)\widehat{f}(0)} I (0,0,0,0,n_1,-n_1) \bigg)\\
   & + 9 \sum_{D\in A_{s,t}\setminus\lbrace 0 \rbrace} \;  \; |\widehat{f}(D)|^4|\widehat{f}(-D)|^2 I(D,D,-D,D,D,-D)\\
   & + |\widehat{f}(0)|^6 I (0,0,0,0,0,0)\\
   = & 9  \sum_{\substack{D\in A_t \\ n_1,n_2\in A_s\setminus \lbrace \pm D \rbrace \\ |n_1|\neq |n_2|}} \widehat{f}(D)\widehat{f}(n_1)\widehat{f}(-n_1)\overline{\widehat{f}(D)\widehat{f}(n_2)\widehat{f}(-n_2)} I(D,n_1,-n_1,D,n_2,-n_2)\\
   & + 9 \; \sum_{\substack{D\in A_t \\ n_1\in A_s\setminus \lbrace \pm D \rbrace}} (2-\delta_{n_1=0}) |\widehat{f}(D)|^2|\widehat{f}(n_1)|^2|\widehat{f}(-n_1)|^2 I (D,D, n_1,-n_1, n_1,-n_1) \\
   & + 18 \, \mathfrak{Re} \bigg( \sum_{\substack{D\in A_{s,t}\setminus \lbrace 0 \rbrace \\ n_1\in A_s\setminus \lbrace \pm D \rbrace} }  \; \widehat{f}(D)\widehat{f}(D)\widehat{f}(-D)
   \overline{\widehat{f}(n_1)\widehat{f}(-n_1)\widehat{f}(D)} I(D,D,-D,D, n_1,-n_1) \bigg) \\
   & + 6 \, \mathfrak{Re}\bigg(\widehat{f}(0)\widehat{f}(0)\widehat{f}(0)\;  \sum_{n_1\in A_s\setminus \lbrace 0 \rbrace} \overline{\widehat{f}(n_1)\widehat{f}(-n_1)\widehat{f}(0)} I (0,0,0,0,n_1,-n_1) \bigg)\\
   & + 9 \sum_{D\in A_{s,t}\setminus\lbrace 0 \rbrace} \;  \; |\widehat{f}(D)|^4|\widehat{f}(-D)|^2 I(D,D,-D,D,D,-D)\\
   & + |\widehat{f}(0)|^6 I (0,0,0,0,0,0)
\end{align*}
Then {using the identity $J_{-n}=(-1)^nJ_n$ and} by the triangle inequality we obtain
\begin{align*}
    (II)\leq & \; 9  \sum_{\substack{D\in A_t \\ n_1,n_2\in A_s\setminus \lbrace \pm D \rbrace \\ |n_1|\neq |n_2|}} |\widehat{f}(D)|^2|\widehat{f}(n_1)\widehat{f}(-n_1)||{\widehat{f}(n_2)\widehat{f}(-n_2)}|\I(D,n_1,n_2)\\
   & + 9 \; \sum_{\substack{D\in A_t \\ n_1\in A_s\setminus \lbrace \pm D \rbrace}} (2-\delta_{n_1=0}) |\widehat{f}(D)|^2|\widehat{f}(n_1)|^2|\widehat{f}(-n_1)|^2 \I (D,n_1,n_1)\\
   & + 18 \; 
   \sum_{\substack{D\in A_{s,t} \setminus \{ 0 \}\\ n_1\in A_s \setminus \lbrace \pm D \rbrace} }|\widehat{f}(D)|^2 |\widehat{f}(D)\widehat{f}(-D)||\overline{\widehat{f}(n_1)\widehat{f}(-n_1)}| \I(D,D,n_1)\\
   & +6  \; 
   \sum_{n_1\in A_s \setminus \lbrace 0 \rbrace} |\widehat{f}(0)|^2|\widehat{f}(0)\widehat{f}(0)| |\overline{\widehat{f}(n_1)\widehat{f}(-n_1)}| \I(0,0,n_1) 
   \\
   & + 9 \; \sum_{D\in A_{s,t} \setminus\lbrace 0 \rbrace}  |\widehat{f}(D)|^4|\widehat{f}(-D)|^2 \I(D,D,D)\\
   & +  |\widehat{f}(0)|^6 \I(0,0,0) ~.
\end{align*}
Next, by using the known inequalities 
$$rs \leq \tfrac12 r^2 + \tfrac12 s^2 \text{ and }r^3s\leq \tfrac{5}{8}r^4 + \tfrac{1}{8}s^4+\tfrac{1}{4}r^2s^2
$$
for $r,s>0$, we further get the following inequality
\begin{align*}
    (II)\leq & \; 9\;  \sum_{\substack{D\in A_t \\ n_1,n_2\in A_s\setminus \lbrace \pm D \rbrace \\ |n_1|\neq |n_2|}}|\widehat{f}(D)|^2\bigg(\tfrac{|\widehat{f}(n_1)|^2+|\widehat{f}(-n_1)|^2}{2}\bigg)
    \bigg(\tfrac{|\widehat{f}(n_2)|^2+|\widehat{f}(-n_2)|^2}{2}\bigg) \I(D,n_1,n_2)\\
   & + 9 \; \sum_{\substack{D\in A_t \\ n_1\in A_s\setminus \lbrace \pm D \rbrace}} (2-\delta_{n_1=0}) |\widehat{f}(D)|^2|\widehat{f}(n_1)|^2|\widehat{f}(-n_1)|^2 \I (D,n_1,n_1)\\
   & + 18  \sum_{\substack{D\in A_{s,t} \setminus \{ 0 \} \\ n_1\in A_s \setminus \lbrace \pm D,0 \rbrace} } |\widehat{f}(D)|^2 \bigg(\tfrac{|\widehat{f}(D)|^2+|\widehat{f}(-D)|^2}{2}\bigg)\bigg(\tfrac{|\widehat{f}(n_1)|^2+|\widehat{f}(-n_1)|^2}{2}\bigg) \I(D,D,n_1)\\
   & + 18  \; 
   \sum_{D\in A_{s,t}\setminus \lbrace 0 \rbrace} \; \bigg( \tfrac{5}{8}|\widehat{f}(D)|^4 +\tfrac{1}{8}|\widehat{f}(-D)|^4 + \tfrac{1}{4}|\widehat{f}(D)\widehat{f}(-D)|^2 \bigg)|\widehat{f}(0)|^2 \I(D,D,0)\\
   & +6  \; 
   \sum_{n_1\in A_s \setminus \lbrace 0 \rbrace} |\widehat{f}(0)|^4\bigg(\tfrac{|\widehat{f}(n_1)|^2+|\widehat{f}(-n_1)|^2}{2}\bigg) \I(0,0,n_1) 
   \\
   & + 9 \sum_{D\in A_{s,t} \setminus\lbrace 0 \rbrace} \; |\widehat{f}(D)|^4|\widehat{f}(-D)|^2 \I(D,D,D)\\
   & + |\widehat{f}(0)|^6 \I(0,0,0) \\
     = & \; 9\;  \sum_{\substack{D\in A_t \\ n_1,n_2\in A_s\setminus \lbrace \pm D \rbrace \\ |n_1|\neq |n_2|}}|\widehat{f}(D)|^2|\widehat{f}(n_1)|^2|\widehat{f}(n_2)|^2 \I(D,n_1,n_2)\\
  & + 9 \; \sum_{\substack{D\in A_t  \\ n_1\in A_s \setminus \lbrace \pm D \rbrace}} (2-\delta_{n_1=0}) |\widehat{f}(D)|^2|\widehat{f}(n_1)|^2|\widehat{f}(-n_1)|^2 \I (D, n_1,n_1)\\
   & + 18 \; 
   \sum_{\substack{D\in A_{s,t} \setminus \lbrace 0 \rbrace \\ n_1\in A_s \setminus \{ \pm D, \, 0 \}}} |\widehat{f}(D)|^2 \bigg(\tfrac{|\widehat{f}(D)|^2+|\widehat{f}(-D)|^2}{2}\bigg)|\widehat{f}(n_1)|^2 \I(D,D,n_1)\\
   & + \frac{27}{2} \; \sum_{D\in A_{s,t} \setminus \lbrace 0 \rbrace} \; |\widehat{f}(D)|^4 |\widehat{f}(0)|^2\I(D,D,0)\\
   & + \frac{9}{2}  \; \sum_{D\in A_{s,t}\setminus \lbrace 0 \rbrace}  \; |\widehat{f}(D)|^2 |\widehat{f}(-D)|^2 |\widehat{f}(0)|^2 \I(D,D,0)\\
    & +6 \; 
   \sum_{n_1\in A_{s,t} \setminus \lbrace 0 \rbrace} |\widehat{f}(0)|^4|\widehat{f}(n_1)|^2 \I(0,0,n_1) 
   \\
   & + 9 \sum_{D\in A_{s,t} \setminus\lbrace 0 \rbrace} |\widehat{f}(D)|^4|\widehat{f}(-D)|^2 \I(D,D,D)\\
   & +  |\widehat{f}(0)|^6 \I(0,0,0).
\end{align*}
We note that such inequalities hold with equality whenever $-A=A$ and $\overline{\ft f(-n)}=(-1)^n\ft f(n)$ 
{which is the case, for instance, if $f$ is real-valued and antipodally symmetric}. Now we sum everything together {and replace $A_t$ by $A$ (observing that $A_t=A$ if $A_s\neq \emptyset$ and $A_t=\emptyset$ if $A_s=\emptyset$)} to obtain the following upper bound
\begin{align*}
    (2\pi)^{-7}||\widehat{f\sigma}||_{L^6(\mathbb{R}^2)}^6 = & \; (I)+(II)\\
    \leq & \sum_{\substack{n_1,n_2,n_3\in A \\ |n_i|\neq |n_j| \, \forall i,j\in \lbrace 1,2,3 \rbrace, \, i\neq j}} \; (6+9\delta_{n_1, n_2\in A_s}) \; |\widehat{f}(n_1)|^2|\widehat{f}(n_2)|^2|\widehat{f}(n_3)|^2  \I(n_1,n_2,n_3)\\
  & + 9 \; \sum_{\substack{ n_1\in A_s, \;n_3\in A \\ |n_1|\neq |n_3|}} (2-\delta_{n_1=0}) |\widehat{f}(n_3)|^2|\widehat{f}(n_1)|^2|\widehat{f}(-n_1)|^2 \I (n_3,n_1,n_1)\\
   & + 9 \; 
   \sum_{\substack{n_1,n_3\in A_s \setminus \lbrace 0 \rbrace \\  |n_1|\neq|n_3|}} |\widehat{f}(n_3)|^2|\widehat{f}(-n_3)|^2|\widehat{f}(n_1)|^2 \I(n_3,n_3,n_1)\\   
& + 9 \sum_{\substack{n_1 \in A\setminus\lbrace 0\rbrace, \, n_3\in A\\ |n_1|\neq |n_3| }} (1+\delta_{n_1,n_3\in A_s\setminus\{0\}}) |\widehat{f}(n_1)|^4|\widehat{f}(n_3)|^2  \I(n_1,n_1,n_3)\\
 & + \frac{27}{2} \; \sum_{n_3\in A_s\setminus \lbrace 0 \rbrace}  \; |\widehat{f}(n_3)|^4 |\widehat{f}(0)|^2 \I(n_3,n_3,0)\\
   & + \frac{9}{2} \; \sum_{n_3\in A_s\setminus \lbrace 0 \rbrace}   \; |\widehat{f}(n_3)|^2 |\widehat{f}(-n_3)|^2 |\widehat{f}(0)|^2 \I(n_3,n_3,0)\\
    & +6 
   \sum_{n_1\in A_s \setminus \lbrace 0 \rbrace} |\widehat{f}(0)|^4|\widehat{f}(n_1)|^2 \I(0,0,n_1)
   \\
   & + 9 \sum_{n_3\in A_s \setminus\lbrace 0 \rbrace} \; |\widehat{f}(n_3)|^4|\widehat{f}(-n_3)|^2 \I(n_3,n_3,n_3)\\
   & +   \sum_{n_1\in A\setminus\lbrace 0\rbrace}  |\widehat{f}(n_1)|^6 \I(n_1,n_1,n_1)\\
   & + |\widehat{f}(0)|^6 \I(0,0,0).
\end{align*}
Next we observe that we may write $(2\pi)^{-3}||f| |_{L^2(\mathbb{S}^1)}^6$ as 
\begin{align*}
    (2\pi)^{-3}||f||_{L^2(\mathbb{S}^1)}^6 = & 
    \sum_{\substack{n_1,n_2,n_3\in A \\ |n_i|\neq |n_j| \, \forall i,j\in \lbrace 1,2,3 \rbrace, \, i\neq j}} \;  |\widehat{f}(n_1)|^2|\widehat{f}(n_2)|^2|\widehat{f}(n_3)|^2\\
    & + 3 \sum_{\substack{n_1\in A_s\setminus\lbrace 0 \rbrace, \; n_3\in A\setminus \lbrace 0 \rbrace \\ |n_1|\neq |n_3|}}  \; |\widehat{f}(n_1)|^2|\widehat{f}(-n_1)|^2|\widehat{f}(n_3)|^2\\
    & +  3\sum_{\substack{n_1\in A\setminus\lbrace 0 \rbrace, \; n_3\in A\setminus \lbrace 0 \rbrace \\ |n_1|\neq |n_3|}}  \; |\widehat{f}(n_1)|^4|\widehat{f}(n_3)|^2 \\
    & +3\sum_{n_1\in A\setminus\lbrace 0 \rbrace} \; |\widehat{f}(n_1)|^4|\widehat{f}(0)|^2\\
    & +3\sum_{n_1\in A_s\setminus\lbrace 0 \rbrace} \;  |\widehat{f}(n_1)|^2|\widehat{f}(-n_1)|^2|\widehat{f}(0)|^2\\
    & +3 \sum_{n_1\in A\setminus\lbrace 0 \rbrace} \; |\widehat{f}(0)|^4|\widehat{f}(n_1)|^2\\
    & + 3 \sum_{n_1\in A_s\setminus\lbrace 0 \rbrace}  \; |\widehat{f}(-n_1)|^2|\widehat{f}(n_1)|^4\\
     & + \sum_{n_1\in A\setminus\lbrace 0 \rbrace} |\widehat{f}(n_1)|^6\\
    & + |\widehat{f}(0)|^6.
\end{align*}
Then by comparison of coefficients (going from bottom to top), the inequality $||\widehat{f\sigma}||_{L^6(\mathbb{R}^2)}^6\leq (2\pi)^4 \I(0,0,0) ||f||_{L^2(\mathbb{S}^1)}^6$ would follow if  
\begin{align*}
& \bullet 3 \I(n,n,n) \leq \I(0,0,0),  \quad 5\I(0,0,n) \leq \I(0,0,0), \quad \tfrac{15}2 \I(n,n,0) \leq \I(0,0,0) \ \ (n>0)\\
& \bullet 9 \I(n,n,m) \leq \I(0,0,0), \quad 15 \I(n,m,\ell) \leq \I(0,0,0) \ \ ( n,m,\ell\geq 0 \text{ distinct}).
\end{align*}
All of them follow easily by Lemma \ref{lemma_bounds_on_bessel} except that the last two above are actually false, the exceptions being $(n,m)=(1,2)$ and $(n,m,\ell)=(3,2,0)$ respectively. However, note that the inequality $9 \I(1,1,2) \leq \I(0,0,0)$ is only needed if $\{1,2\}\subset A_s\setminus\{0\}$ which is impossible since { $1+1+1=2+2-1$} and so $A$ would not be a $\P(3)$-set. We conclude that in fact we only need $6 \I(1,1,2) \leq \I(0,0,0)$, which is true by Lemma \ref{lemma_bounds_on_bessel}. Similarly, the  inequality $15 \I(3,2,0)<\I(0,0,0)$ is only required if $\{3a, 2b,0\}\subset A$ for some $a,b\in \{\pm 1\}$ and either $2$ or $3$ also belong to $-A$. This cannot be true since $3+3+0=2+2+2$, and $A$ would not be a $\P(3)$-set. We conclude that the inequality we actually need is $6 \I(3,2,0)<\I(0,0,0)$, which follows from Lemma \ref{lemma_bounds_on_bessel}. This finishes the proof.\qed

\section{A further example of application}\label{extra_section}
Arguments similar to those in the previous section can be used to establish other sharp extension inequalities for functions in $L^2(\mathbb{S}^1)$ whose spectrum satisfies property $\mathrm{P}(h)$ for some suitable $h$. In this section we provide a further example of application for the case of the $L^2(\mathbb{S}^1)$ to $L^6_{rad}L^4_{ang}(\mathbb{R}^2)$ Fourier extension estimates. The case  of sharp $L^2(\mathbb{S}^1)$ to $L^6_{rad}L^2_{ang}(\mathbb{R}^2)$ Fourier extension estimates has been studied in { \cite{foschi2017some}, see also \cite{carneiro2019sharp}}.

\begin{theorem}
Let $f\in L^2(\mathbb{S}^1)$ be such that its spectrum $A$ satisfies property $\mathrm{P}(2)$. Then
$$||\widehat{f\sigma}||_{L^6_{rad}L^4_{ang}}^6\leq (2\pi)^{9/2}\bigg(\int_0^\infty J_0^6(r)rdr \bigg) ||f||_{L^2(\mathbb{S}^1)}^6.$$
The inequality is sharp and equality is attained if and only if $f$ is constant.
\end{theorem}

\proof Without loss of generality we can assume that $\sum_{n\in A}|\widehat{f}(n)|^2 =1$. Using Hecke-Bochner formula we have that  
\begin{align*}
    ||\widehat{f\sigma} & ||_{L^6_{rad}L^4_{ang}(\mathbb{R}^2)}^{ 6}= \int_0^\infty \bigg(\int_{\mathbb{S}^1} |\widehat{f\sigma}(r\omega)|^4 d\sigma(\omega) \bigg)^{6/4} rdr\\
    & =  (2\pi)^{15/2}\int_0^\infty \bigg(\sum_{\substack{n_1,n_2,n_3,n_4\in A\\ n_1+n_2=n_3+n_4}}
    \widehat{f}(n_1)\widehat{f}(n_2) \overline{\widehat{f}(n_3)\widehat{f}(n_4)}
    J_{n_1}(r)J_{n_2}(r)J_{n_3}(r)J_{n_4}(r) \bigg)^{3/2} rdr.
\end{align*}
 Now we use the fact that $A$ satisfies property $\mathrm{P}(2)$ to rewrite the sum in the integral as follows.
 \begin{align*}
  &   \sum_{\substack{n_1,n_2,n_3,n_4\in A\\ n_1+n_2=n_3+n_4}} 
    \widehat{f}(n_1)\widehat{f}(n_2)  \overline{\widehat{f}(n_3)\widehat{f}(n_4)} 
    J_{n_1}J_{n_2}J_{n_3}J_{n_4} \\ 
    = & \sum_{D\in A^2} \; \sum_{\substack{n_1,n_2,n_3,n_4\in A \\ n_1+n_2=D\\ n_3+n_4=D}}
    \widehat{f}(n_1)\widehat{f}(n_2) \overline{\widehat{f}(n_3)\widehat{f}(n_4)}
    J_{n_1}J_{n_2}J_{n_3}J_{n_4}\\
    = & \sum_{\substack{n_1,n_2\in A \\ n_1\neq -n_2}}\tau(n_1,n_2) |\widehat{f}(n_1)|^2|\widehat{f}(n_2)|^2  
    J_{n_1}^2J_{n_2}^2 \\
    & +\sum_{n_1,n_2\in A_s} \widehat{f}(n_1)\widehat{f}(-n_1)\overline{\widehat{f}(n_2)\widehat{f}(-n_2)} J_{n_1}J_{-n_1}J_{n_2}J_{-n_2} \\
    \leq & \sum_{\substack{n_1,n_2\in A \\ n_1\neq -n_2}}\tau(n_1,n_2) |\widehat{f}(n_1)|^2|\widehat{f}(n_2)|^2  
    J_{n_1}^2J_{n_2}^2 \\
    & + \sum_{n_1,n_2\in A_s} \bigg(\frac{|\widehat{f}(n_1)|^2J_{n_1}^2 + |\widehat{f}(-n_1)|^2J_{-n_1}^2}{2} \bigg)\bigg( \frac{|\widehat{f}(n_2)|^2J_{n_2}^2 + |\widehat{f}(-n_2)|^2J_{-n_2}^2}{2} \bigg)\\
    = & \sum_{n_1,n_2\in A} (\tau(n_1,n_2)\delta_{n_1\neq -n_2} + \delta_{n_1, n_2\in A_s})|\widehat{f}(n_1)|^2|\widehat{f}(n_2)|^2  
    J_{n_1}^2J_{n_2}^2
 \end{align*}
where $\tau(n_1,n_2)=1+\delta_{n_1\neq n_2}$ is the number of permutations of $(n_1,n_2)$. Hence by Jensen's inequality we obtain 
\begin{align*}
  & (2\pi)^{-15/2}  ||\widehat{f\sigma}||_{L^6_{rad}L^4_{ang}(\mathbb{R}^2)}^{ 6} \\ & \leq  \int_0^\infty \bigg(\sum_{n_1,n_2\in A}
    (\tau(n_1,n_2)\delta_{n_1\neq -n_2} + \delta_{n_1, n_2\in A_s})|\widehat{f}(n_1)|^2|\widehat{f}(n_2)|^2  
    J_{n_1}^2(r)J_{n_2}^2(r)\bigg)^{3/2} rdr\\
    \leq &  \int_0^\infty \sum_{n_1,n_2\in A} (\tau(n_1,n_2)\delta_{n_1\neq -n_2} + \delta_{n_1, n_2\in A_s})^{3/2} |J_{n_1}(r)|^3|J_{n_2}(r)|^3 |\widehat{f}(n_1)|^2|\widehat{f}(n_2)|^2 rdr\\
    = &  \sum_{n_1,n_2\in A} (\tau(n_1,n_2)\delta_{n_1\neq -n_2} + \delta_{n_1, n_2\in A_s})^{3/2} \bigg( \int_0^\infty |J_{n_1}(r)|^3|J_{n_2}(r)|^3  rdr \bigg) |\widehat{f}(n_1)|^2|\widehat{f}(n_2)|^2.
\end{align*}
To conclude we need the following estimates on integrals involving the products of six Bessel functions. First, we observe that for all $m\neq \ell$, $m,\ell\geq 0$ it holds that
$$ 3^{3/2} \int_0^\infty |J_{m}(r)|^3|J_{\ell}(r)|^3 rdr  { <} \int_0^\infty J_0^6 (r) rdr ~. $$
In fact, by H{\" o}lder inequality and by the estimates in Lemma \ref{lemma_bounds_on_bessel} we have that
\begin{align*}
    \int_0^\infty |J_{m}(r)|^3|J_{\ell}(r)|^3 rdr \; \leq & \; \bigg( \int_0^\infty |J_m(r)|^4 |J_\ell(r)|^2 rdr \bigg)^{1/2}\bigg(\int_0^\infty|J_\ell(r)|^4|J_m(r)|^2 rdr \bigg)^{1/2}\\
    \leq & \; ( \tfrac{\sqrt{2}}{5\sqrt{3}}\delta_{m \ell =0}\, +\tfrac{1}{9} \delta_{\lbrace m,\ell \rbrace \neq \lbrace 1,2 \rbrace,\, m\ell\neq 0}\, +\tfrac{1}{3\sqrt{6}}\delta_{\lbrace m, \ell\rbrace=\lbrace 1,2 \rbrace}\, )\bigg( \int_0^\infty J^6_0(r) rdr \bigg)^{1/2}\\
    < & \; 3^{-3/2} \int_0^\infty J_0^6(r) rdr.
\end{align*}
The second estimate that we need is the following: for all $n >0$ it holds that 
$$ 2^{3/2} \int_0^\infty J_n^6 (r) rdr < \int_0^\infty J_0^6(r) rdr $$
 which follows from Lemma \ref{lemma_bounds_on_bessel} again. Hence, we conclude that 
 \allowdisplaybreaks
 \begin{align*}
  ||\widehat{f\sigma}||_{L^6_{rad}L^4_{ang}(\mathbb{R}^2)}^{ 6} \leq  &  \;   (2\pi)^{15/2}  \bigg(\int_0^\infty J^6_0(r)rdr \bigg) \sum_{n_1,n_2\in A} |\widehat{f}(n_1)|^2|\widehat{f}(n_2)|^2\\
     = &  \; (2\pi)^{9/2} \bigg(\int_0^\infty J^6_0(r)rdr \bigg) ||f||_{L^2(\mathbb{S}^1)}^6 .
 \end{align*}
 Equality is attained  if and only if $f$ is constant. \qed

{
 \section*{Acknowledgements}
The authors are grateful to Christoph Thiele for stimulating discussions. This work was developed while F.G. was an Akademischer Rat at the Universit\"at Bonn and he acknowledges support from the Deutsche Forschungsgemeinschaft through the Collaborative Research Center 1060.}

\begin{center}

\end{center}

\end{document}